\documentclass{amsart}

\usepackage{helvet}

\usepackage{amsmath,amsxtra,amsthm,stmaryrd,amssymb,xr,color}
\usepackage{tikz-cd}
\usepackage[all]{xy}

\newtheorem{theorem}{Theorem}[section]
\newtheorem{lemma}[theorem]{Lemma}
\newtheorem{conj}[theorem]{Conjecture}
\newtheorem{proposition}[theorem]{Proposition}

\newtheorem{remark}[theorem]{Remark}

\newcommand{\Gal}{\operatorname{Gal}}
\newcommand{\Fil}{\operatorname{Fil}}
\newcommand{\DD}{\mathbb{D}}

\newcommand{\QQ}{\mathbb{Q}}
\newcommand{\Qp}{\mathbb{Q}_p}
\newcommand{\Zp}{\mathbb{Z}_p}
\newcommand{\Dp}{\mathbb{D}_p}
\newcommand{\ZZ}{\mathbb{Z}}

\newcommand{\cM}{\mathcal{M}}

\newcommand{\rank}{\mathrm{rank}}

\newcommand{\p}{\mathfrak{p}}

\newcommand{\vp}{\varphi}
\newcommand{\cL}{\mathcal{L}}
\newcommand{\cH}{\mathcal{H}}

\newcommand{\HIw}{H^1_{\mathrm{Iw}}}

\newcommand{\GL}{\mathrm{GL}}

\newcommand{\col}{\mathrm{Col}}
\newcommand{\image}{\operatorname{Im}}
\newcommand{\coker}{\operatorname{coker}}

\newcommand{\loc}{\mathrm{loc}}

\newcommand{\Hom}{\mathrm{Hom}}
\newcommand{\Sel}{\mathrm{Sel}}

\newcommand{\Ind}{\mathrm{Ind}}
\newcommand{\lb}{\llbracket}
\newcommand{\rb}{\rrbracket}
\newcommand{\ZpX}{\Zp\lb X\rb}

\newtheorem{lettertheorem}{Theorem}

\begin{document}

\title{Functional equations for multi-signed Selmer groups}

\begin{abstract}
We study the functional equation for the multi-signed Selmer groups for non-ordinary motives whose Hodge-Tate weights are $0$ and $1$, defined by  B\"uy\"ukboduk and the first named author in \cite{buyukboduklei0}. This generalizes simultaneously Greenberg's result for ordinary motives in \cite{greenberg89} and Kim's result for supersingular elliptic curves in \cite{kim08}.
\end{abstract}

\author{Antonio Lei}
\address{Antonio Lei\newline
D\'epartement de Math\'ematiques et de Statistique\\
Universit\'e Laval, Pavillion Alexandre-Vachon\\
1045 Avenue de la M\'edecine\\
Qu\'ebec, QC\\
Canada G1V 0A6}
\email{antonio.lei@mat.ulaval.ca}

\author{Gautier Ponsinet}
\address{Gautier Ponsinet\newline
D\'epartement de Math\'ematiques et de Statistique\\
Universit\'e Laval, Pavillion Alexandre-Vachon\\
1045 Avenue de la M\'edecine\\
Qu\'ebec, QC\\
Canada G1V 0A6}
\email{gautier.ponsinet.1@ulaval.ca}

\thanks{The authors' research is supported by the Discovery Grants Program 05710 of NSERC}

\maketitle

\section{Introduction}

Let $p$ be an odd prime and $F$ a fixed number field in which all primes above $p$ are unramified. Furthermore, we assume that $F/\QQ$ is an abelian extension with $[F:\QQ]$ coprime to $p$.  Let $\cM_{/F}$ be a motive defined over $F$. In \cite{greenberg89},
Greenberg developed an Iwasawa theory for a $p$-ordinary motive $\cM$. In particular, he showed that its $p$-Selmer group over the $\Zp$-cyclotomic extension $F_\infty$ of $F$ satisfies a functional equation, namely,
\[
\Sel_p(\cM/F_\infty)\sim\Sel_p(\widehat{\cM}/F_\infty)^\iota,
\]
where  $\widehat{\cM}$ is the dual motive of $\cM$, $A\sim B$ means that $A$ and $B$ are pseudo-isomorphic as $\Zp\lb\Gal(F_\infty/F)\rb$-modules and $\iota$ is the involution on the Iwasawa algebra that sends $\sigma$ to $\sigma^{-1}$ for $\sigma\in\Gal(F_\infty/F)$. This result has been generalized to elliptic curves with supersingular reduction at $p$ by Kim \cite{kim08}. That is, 
\[
\Sel_p^\pm(E/F_\infty)\sim\Sel_p^\pm(E/F_\infty)^\iota,
\]
where $E$ is an elliptic curve with $a_p(E)=0$ and $\Sel_p^\pm(E/F_\infty)$ are Kobayashi's plus and minus Selmer groups from \cite{kobayashi03}.

Let $\cM_p$ be the $p$-adic realization of $\cM$. Let $g:=\dim_{\Qp}(\Ind_{F/\QQ}\cM_p)$ and $g_\pm:=\dim_{\Qp}(\Ind_{F/\QQ}\cM_p)^{c=\pm 1}$, where $c$ is the complex conjugation. We shall assume that $\cM_p$ satisfies:
\begin{itemize}
\item[(H.HT)] The Hodge-Tate weights of $\cM_p$ are either $0$ or $1$;
\item[(H.crys)] $\cM_p$ is crystalline at all primes $\p|p$ of $F$;
\item[(H.Frob)] The slopes of the Frobenius on the Dieudonn\'e module $\DD_p(\cM_p)$ lie inside $(0,-1]$ and that $1$ is not an eigenvalue;
\item[(H.P)]  $g_{+}=g_-$ and   $\dim_{\Qp}\Fil^0\Dp(\cM_p)=g_-$.
\end{itemize}
Note that (H.HT) would essentially force us to restrict our attention to abelian varieties. In this case,  (H.crys) means that it has good reduction at all primes above $p$. If furthermore it has  supersingular reduction at $p$, then the Frobenius has constant slope $-1/2$ on the Dieudonn\'e module, which implies (H.Frob). In general, if $\cM$ is irreducible and pure, then Tate's conjecture for $\cM_p$ would imply (H.P) (c.f. \cite[Remark~1.3]{buyukboduklei0}).

Under these hypotheses, a family of multi-signed Selmer groups $\Sel_I(\cM/F_\infty)$ has been defined using the theory of Coleman maps in \cite{buyukboduklei0} (see \S\ref{S:defnSel} below for a review). Here, $I$ corresponds to a choice of a sub-basis of $\DD_p(\cM_p)$. These Selmer groups generalize Kobayashi's plus and minus Selmer groups in \cite{kobayashi03} and Sprung's $\#/\flat$-Selmer groups in \cite{sprung09} for supersingular elliptic curves.

The main result of this article is the following theorem.
\begin{lettertheorem}[Theorem~\ref{thm:main}]\label{thm:A}Suppose that $H^0(F_\infty,\cM_p^\vee)=0$, where $\cM_p^\vee$ denotes the Pontryagin dual of $\cM_p$. Furthermore, suppose  that $I$ is a choice of basis satisfying $\#I=g_+=g_-$. Then, we have the functional equation
\[\Sel_I(\cM/F_\infty)\sim\Sel_{I^c}(\widehat{\cM}/F_\infty)^\iota,\]
where $I^c$ corresponds the sub-basis of $\DD_p(\cM_p^\vee)$ that is dual to $I$.
\end{lettertheorem}

In \cite{buyukboduklei0}, the multi-signed Selmer groups are conjectured to be cotorsion over $\Zp\lb\Gal(F_\infty/F)\rb$.  Furthermore, a main conjecture relating them to some multi-signed $p$-adic $L$-functions has been formulated. We remark that Theorem~\ref{thm:A} holds even without the cotorsionness of the Selmer groups. In fact, it gives evidence to the main conjecture because the $L$-functions are expected to satisfy a similar functional equation.

The structure of this paper is as follows. We first review the construction of multi-signed Selmer groups of \cite{buyukboduklei0} in \S\ref{S:defnSel}. We then show that the local conditions of these Selmer groups satisfy an orthogonality condition and a control theorem in \S\ref{S:orthogonal} and \S\ref{S:control} respectively. These two results allow us to conclude the proof of Theorem~\ref{thm:A} using techniques of Greenberg and Kim, which is the content of \S\ref{S:proof}. Finally, we explain in \S\ref{S:AV} that our result applies to abelian varieties and allows us to obtain the following theorem.

\begin{lettertheorem}[Theorem~\ref{thm:AV}]
Let $A$ be a  $g$-dimensional abelian variety defined over $F$ with supersingular reduction at all primes above $p$. Then, for any $I$ such that $\# I = g$, the dual Selmer groups $\Sel_I(A/F(\mu_{p^\infty}))^\vee$ and $\Sel_{I^c}(A^\vee/F(\mu_{p^\infty}))^{\vee,\iota}$ are pseudo-isomorphic $\Zp\lb\Gal(F_\infty/F)\rb$-modules.
\end{lettertheorem}

\section*{Acknowledgment}
The authors would like to thank K\^az\i m B\"uy\"ukboduk, Daniel Delbourgo and Florian Sprung for interesting discussion during the preparation of this article. We would also like to thank the anonymous referee for very useful suggestions and comments which led to many improvements on the presentation of the article.

\section{Review of the construction of multi-signed Selmer groups}\label{S:defnSel}

Let us first introduce some notation. Let $\Gamma$ be the Galois group $\Gal(\Qp(\mu_{p^\infty})/\Qp)$. Given any unramified extension $K$ of $\Qp$, we shall abuse notation and write $\Gamma$ for the Galois group $\Gal(K(\mu_{p^\infty})/K)$ as well. 
We may decompose $\Gamma$ as $\Delta\times\Gamma^1$, where $\Delta$ is cyclic of order $p-1$ and $\Gamma^1 = \overline{\langle\gamma\rangle}$ is isomorphic to the additive group $\Zp$. We shall denote by $K_\infty$ the subfield of $K(\mu_{p^\infty})$ fixed by $\Delta$, thus $K_\infty$ is a $\Gamma^1$-extension of $K$ and we also set $K_n$ the subextention of $K_\infty$ such that $\Gamma^1/\Gamma^1_n \simeq \ZZ/p^n\ZZ$, where $\Gamma^1_n := \Gal(K_\infty/K_n)$.

We write $\Lambda$ for the Iwasawa algebra $\Zp\lb\Gamma\rb$. We may identify it with the set of power series $\sum_{n\ge0,\sigma\in\Delta}a_{n,\sigma}\cdot\sigma\cdot(\gamma-1)^n$ where $a_{n,\sigma}\in\Zp$. We  denote by $\iota :\Lambda \rightarrow \Lambda$ the involution induced by $\sigma \mapsto \sigma^{-1}$ for $\sigma \in \Gamma$ and we identify $\gamma-1$ with the indeterminate $X$. Similarly, we write $\Lambda^1 := \Zp\lb \Gamma^1\rb$, and we identify it with $\Zp\lb X\rb$ in the same way.

Let $M$ be a $\Lambda$-module, $\eta$ a Dirichlet character modulo $p$. We write $e_\eta=\frac{1}{p-1}\sum_{\sigma\in \Delta}\eta(\sigma)^{-1}\sigma\in\Zp[\Delta]$ for the idempotent corresponding to $\eta$. The $\eta$-isotypic component of $M$ is defined to be $e_\eta\cdot M$ and denoted by $M^\eta$. Note that we may regard $M^\eta$ as a $\Zp\lb X\rb$-module. We  denote by $M^\iota$ the $\Lambda$-module whose action is twisted by $\iota$. Also, we write $M^\vee := \Hom(M , \Qp/\Zp)$ the Pontryagin dual of $M$.

We define $\cH$ to be the set of elements $\sum_{n\ge0,\sigma\in\Delta}a_{n,\sigma}\cdot\sigma\cdot(\gamma-1)^n$ where $a_{n,\sigma}\in\Qp$ are such that the power series $\sum_{n\ge0}a_{n,\sigma}X^n$ converges on the open unit disc for all $\sigma\in\Delta$.

Suppose now that $\cM$ is a motive satisfying (H.HT), (H.crys) and (H.Frob) as in the introduction (but not (H.P) for the time being).  
 Fix a $G_F$ stable $\Zp$-lattice $T$ inside $\cM_p$ and write $T^\dagger$ for the Cartier dual $\Hom(T,\mu_{p^\infty})$. For each prime $\p|p$ of $F$, we write $\DD_{F_\p}(T)$ for the Dieudonn\'e module of the local representation $T|_{G_{F_\p}}$. Let
\[
\Dp(T)=\bigoplus_{\p|p}\DD_{F_\p}(T),
\]
which admits the Frobenius action $\vp:\Dp(T)\rightarrow\Dp(T)\otimes\Qp$ and a filtration $\Fil^\bullet\Dp(T)$. We assume:
\begin{itemize}
\item[(H.P')] $\rank_{\Zp}\Fil^0\Dp(T)=g_-$.
\end{itemize}

Fix a $\Zp$-basis $v_1,\ldots, v_g$ of $\Dp(T)$ such that $v_1,\ldots,v_{g_-}$ generate $\Fil^0\Dp(T)$. By (H.HT), the matrix of $\vp$ with respect to this basis is of the form
\[
C_\vp=C\left(
\begin{array}{c|c}
I_{g_-}&0\\ \hline
0&\frac{1}{p}I_{g_+}
\end{array}
\right),
\]
where $I_{g_\pm}$ denote the identity matrices of rank $g_\pm$ and $C$ is some matrix inside $\GL_g(\Zp)$. As in \cite[Definition~2.4]{buyukboduklei0}, we may define for $n\ge1$, 
\begin{equation}\label{eq:defnmatrix}
C_n=
\left(
\begin{array}{c|c}
I_{g_-}&0\\ \hline
0&\Phi_{p^n}(1+X)I_{g_+}
\end{array}
\right)C^{-1}
\quad\text{and}\quad
M_n=\left(C_\vp\right)^{n+1}C_{n}\cdots C_{1},
\end{equation}
where $\Phi_{p^n}$ denotes the $p^n$-th cyclotomic polynomial. By Proposition~2.5 in \textit{op. cit.}, the sequence $M_n$, $n\ge1$ converges to some $g\times g$ logarithmic matrix over $\cH$, which is denoted by $M_T$.

For each prime  $\p|p$ of $F$, we write
\[
\HIw(F_\p,T)=\varprojlim H^1(F_\p(\mu_{p^n}),T),
\]
where the connecting map is the corestriction maps. This allows us to define
\[
\HIw(F_p,T)=\bigoplus_{\p|p}\HIw(F_\p,T).
\]
The dual of Perrin-Riou's exponential map from \cite{perrinriou94} gives a $\Lambda$-homomorphism
\[
\cL_T:\HIw(F_p,T)\rightarrow \cH\otimes_{\Zp}\Dp(T).
\]
By \cite[Theorem~2.13]{buyukboduklei0}, the logarithmic matrix $M_T$ allows us to decomposes $\cL_T$ into
\begin{equation} \label{decomposition}
\cL_T=\begin{pmatrix}
v_1&\cdots&v_g
\end{pmatrix}\cdot M_T\cdot \begin{pmatrix}
\col_{T,1}\\ \vdots \\ \col_{T,g}
\end{pmatrix},
\end{equation}
where $\col_{T,i}$, $i=1,\ldots,g$ are $\Lambda$-homomorphisms from $\HIw(F_p,T)$ to $\Lambda$, and we shall denote by $\col_T$ the column vector that appears on the right-hand side.
The images of these maps are described in \S2.5 of \textit{op. cit.}. For $I\subset\{1,\ldots,g\}$, we shall write  $\col_{T,I}$ for the direct sum $\bigoplus_{i\in I}\col_{T,i}$. We record the following result (Proposition~2.21 and Corollary~2.22 of \textit{op. cit.}), which we shall make use of later on in this article.

\begin{proposition} \label{secReview:prop1}
Let $I\subset\{1,\ldots,g\}$ and $\eta$ a character on $\Delta$. Then, $\image\left(\col_{T,I}\right)^\eta$ is contained in a free $\ZpX$-module of rank $\#I$ inside $\bigoplus_{i\in I}\ZpX$. Furthermore, this inclusion is of finite index.
\end{proposition}

We recall equally the following result on the kernel of the Coleman maps (Lemma~3.22 of \textit{op. cit.}), which we shall need.
\begin{lemma} \label{secReview:lem1}
For any $I\subset\{1,\ldots,g\}$, the $\Lambda$-module $\ker(\col_{T,I})$ is free of rank $g-\#I$.
\end{lemma}

Using these Coleman maps, the multi-signed Selmer groups are defined in \S3.4 of \textit{op. cit.} For a fixed $I\subset\{1,\ldots,g\}$, with $\#I=g_+$, we define $H^1_{I}(F(\mu_{p^\infty})_p,T^\dagger)$ to be the orthogonal complement of $\ker\col_{T,I}$ under the (semi-)local Tate duality
\[
\HIw(F_p,T)\times H^1(F(\mu_{p^\infty})_p,T^\dagger)\rightarrow\Qp/\Zp,
\]
where $H^1(F(\mu_{p^\infty})_p,T^\dagger)$ denotes $
\bigoplus_{\p|p}H^1(F(\mu_{p^\infty})_\p,T^\dagger)$. Here, the pairing is defined to be the sum of the local Tate pairings for all $\p|p$. The $I$-Selmer group of $T^\dagger$ over $F(\mu_{p^\infty})$ (denoted by $\Sel_I(T^\dagger/F(\mu_{p^\infty}))$) is defined to be
\[
\ker\left(  H^1(F(\mu_{p^\infty}),T^\dagger) \longrightarrow
\bigoplus_{v\nmid p}\frac{H^1(F(\mu_{p^\infty})_v,T^\dagger)}{H^1_f(F(\mu_{p^\infty})_v,T^\dagger)}\oplus\frac{H^1(F(\mu_{p^\infty})_p,T^\dagger)}{H^1_{I}(F(\mu_{p^\infty})_p,T^\dagger)}\right),
\]
where $H^1_f(F(\mu_{p^\infty})_v,T^\dagger)$ for $v\nmid p$ is the unramified subgroup of $H^1(F(\mu_{p^\infty})_v,T^\dagger)$. The even isotypic components of the $I$-Selmer group are predicted to be cotorsion:

\begin{conj}
For all $I$ as above and any even characters $\eta$ on $\Delta$, $\Sel_I(T^\dagger/F(\mu_{p^\infty}))^\eta$ is a cotorsion $\ZpX$-module.
\end{conj}
See \cite[Proposition~3.28]{buyukboduklei0} for a partial result. If we  replace (H.P') by the stronger hypothesis (H.P), then the calculations in \textit{loc. cit.} may be generalized to odd characters $\eta$ on swapping the roles of $g_+$ and $g_-$ throughout (c.f. Remark~1.7 of \textit{op. cit.} for a detailed explanation).

\section{Orthogonality of local conditions}\label{S:orthogonal}

As in \S\ref{S:defnSel}, we assume that (H.HT), (H.crys) and (H.Frob) hold in this section.
Let $x$ be an element of $\HIw(F_\p, T)$, we denote by $x_n$ its component in $H^1(F(\mu_{p^n})_p, T)$. Let $\langle \sim,\sim\rangle_{n}$ be the (semi-)local Tate pairing
$$
\langle \sim,\sim \rangle_{n} : H^1(F(\mu_{p^n})_p, T) \times  H^1(F(\mu_{p^n})_p, T^*(1)) \rightarrow \Zp.
$$
If $x$ and $y$ are elements of $\HIw(F_\p, T)$ and $\HIw(F_\p, T^*(1))$ respectively, the elements
$$
\sum\limits_{\sigma \in \Gamma/\Gamma^{p^n}} \langle x_n, \sigma(y_n) \rangle_n \cdot \sigma \in \Zp[\Gamma/\Gamma^{p^n}]
$$
are compatible  with respect to the natural projection maps for $n\ge0$.
Thus, on taking inverse limit, one can define the Perrin-Riou pairing
$$
\langle \sim,\sim \rangle : \HIw(F_p, T) \times \HIw (F_p, T^*(1))  \rightarrow \Lambda.
$$

There is also a natural pairing
$$
\DD_{F_\p}(T) \times \DD_{F_\p}(T^*(1)) \rightarrow \Zp,
$$
thus a pairing
$$
[\sim,\sim] : \Dp(T) \times \Dp(T^*(1)) \rightarrow \Zp,
$$
for which $\Fil^i\Dp(T^*(1)) = \Fil^{1-i}\Dp(T)^\perp$. 

We denote $C^*_\varphi$ the matrix of the Frobenius on $\Dp(T^*(1))$ with respect to the dual basis $v_1',\ldots,v_g'$. From duality we have the relation
$$
C^*_\varphi   = \frac{1}{p} \cdot ({C_\varphi}^{-1})^t = \left(C^{-1}\right)^t \left(
\begin{array}{c|c}
\frac{1}{p}I_{g_-}&0\\ \hline
0&I_{g_+}
\end{array}
\right).
$$
As in \eqref{eq:defnmatrix}, we may define
\begin{equation}\label{eq:defnmatrixdual}
M_{T^*(1)}   = \varinjlim\limits_n   \left(C^*_\vp\right)^{n+1}C^*_{n}\cdots C^*_{1}, \quad\text{where}\quad C^*_n = \left(
\begin{array}{c|c}
\Phi_{p^n}(1+X)I_{g_-}&0\\ \hline
0&I_{g_+}
\end{array}
\right)C^{t}
\end{equation}
Analogous to \eqref{decomposition}, we have the decomposition
\begin{equation} \label{decompositiondual}
\cL_{T^*(1)}=\begin{pmatrix}
v_1'&\cdots&v_g'
\end{pmatrix}\cdot M_{T^*(1)}\cdot \col_{T^*(1)},
\end{equation}
where $\col_{T^*(1)}$ is the column vector of Coleman maps with respect to the basis $v_1',\ldots,v_g'$.

Perrin-Riou's explicit reciprocity law from \cite[\S3.6]{perrinriou94} (proved by Colmez \cite[\S IX.4]{colmez98}) tells us that
\begin{equation} \label{ERL}
[\cL_T(z),\cL_{T^*(1)}(z')] =-\sigma_{-1}\cdot \ell_0 \cdot\langle  z ,  z'\rangle ,
\end{equation}
where $\sigma_{-1}$ is the unique element of $\Gamma$ of order $2$ and $\ell_0 = \frac{\log \gamma}{\log \chi(\gamma)}$ with $\chi$ being the cyclotomic character. The explicit reciprocity law we stated is slightly different from the one stated in \textit{op. cit}. See  \cite[Theorem~B.6]{loefflerzerbes14} for a proof of the above formulation.

\begin{lemma} \label{secOrth:lem1}
Let $z \in \HIw(F_p, T)$ and $z' \in \HIw (F_p, T^*(1))$. Then
$$
[\cL_T(z), \cL_{T^*(1)}(z')] = \frac{\log(1+X)}{pX} \cdot \col_{T}(z)^t \cdot \col_{T^*(1)}(z').
$$ 
\end{lemma}
\begin{proof}
Since $v_1,\ldots, v_g$ is dual to $v_1',\ldots,v_g'$,  the decompositions \eqref{decomposition} and \eqref{decompositiondual} give
$$
[\cL_T(z), \cL_{T^*(1)}(z')] = \col_{T}(z)^t \cdot {M_T}^t \cdot M_{T^*(1)} \cdot \col_{T^*(1)}(z').
$$
Using \eqref{eq:defnmatrix} and \eqref{eq:defnmatrixdual}, we may compute
$$
\begin{array}{rl}
{M_T}^t \cdot M_{T^*(1)} & = \left(\varinjlim \left(C_\vp\right)^{n+1}C_{n}\cdots C_{1} \right)^t \cdot \left(\varinjlim \left(C^*_\vp\right)^{n+1}C^*_{n}\cdots C^*_{1}\right) \\
& =\varinjlim \left( C_{1}^t \cdots C_{n}^t\left({C_\vp}^t\right)^{n+1}\cdot \left(C^*_\vp\right)^{n+1}C^*_{n}\cdots C^*_{1}\right) \\
& = \lim \frac{1}{p^{n+1}} \prod\limits_{k = 1}^n \Phi_{p^n}(1+X) I_g \\
& = \frac{\log(1+X)}{pX}I_g.
\end{array}
$$
\end{proof}
\begin{lemma} \label{secOrth:lem2}
Let $I \subset \{1, \ldots ,g\}$ and $I^c$ its complement. Then $\ker \col_{T^*(1),I^c}$ is the orthogonal complement of $\ker \col_{T,I}$ with respect to the pairing $\langle \cdot , \cdot \rangle$.
\end{lemma}
\begin{proof}
Let $z \in \HIw(F_p, T)$ and $z' \in \HIw (F_p, T^*(1))$. By the explicit reciprocity law (\ref{ERL}) and Lemma~\ref{secOrth:lem1}, we have
$$
\begin{array}{rl}
\langle z, z'\rangle = 0 & \Leftrightarrow [\cL_T(z), \cL_{T^*(1)}(z')] = 0 \\
& \Leftrightarrow \col_{T}(z)^t \cdot \col_{T^*(1)}(z') = 0 .
\end{array}
$$
Thus, if $z \in \ker \col_{T,I}$,
\begin{equation}\label{eq:orthoequiv}
\langle z, z'\rangle = 0 \Leftrightarrow \sum\limits_{k \notin I} \col_{T,k}(z) \cdot \col_{T^*(1),k}(z') = 0.
\end{equation}
So, $\ker \col_{T^*(1),I^c}$ is included in the orthogonal complement of $\ker \col_{T,I}$. 

Lemma~\ref{secReview:lem1} implies that for all $k \in\{1,\ldots, g\}$, there exists $z_k$ such that 
\[\col_{T,j}(z_k)\begin{cases}
=0 &\text{if $j\in \{1,\ldots, g\}\setminus\{ k\}$,}\\
\ne 0&\text{if $j= k$.}
\end{cases}
\] 
In particular, if $k\notin I$, then such $z_k\in\ker\col_{T,I}$. If $z'\in \left(\ker\col_{T,I}\right)^{\perp}$, then $\langle z_k,z'\rangle=0$. Therefore, \eqref{eq:orthoequiv} tells us that $\col_{T^*(1),k}(z')=0$. Since this is true for all $k\in I^c$, we have $z'\in \ker\col_{T^*(1),I^c}$ as required.
\end{proof}

\section{Control theorem}\label{S:control}

Once again, the hypotheses (H.HT), (H.crys) and (H.Frob) are in effect throughout this section.
We first deal with the case $F=\QQ$ and then explain how to extend the results of this section to general $F$. 
Let $f$ be an irreducible distinguished polynomial in $\Lambda^1$, $e$ and $m$ positive integers. Set 
$$
\begin{array}{ll}
T^\dagger_{f^e} := T^\dagger \otimes_{\Zp} \Lambda^1/(f^e), &  {T^*(1)}^\dagger_{(f^e)^\iota}:={T^*(1)}^\dagger \otimes_{\Zp} \Lambda^1/(f^e)^\iota .\\
\end{array} 
$$
Then $\Hom(T^\dagger_{f^e}[p^m],\mu_{p^m}) \simeq {T^*(1)}^\dagger_{(f^e)^\iota}[p^m]$ and the pairing $T^\dagger[p^m] \times {T^*(1)}[p^m]\rightarrow \mu_{p^m}$ induces a perfect $\Zp$-linear $G_\QQ$-equivariant pairing $T^\dagger_{f^e}[p^m] \times {T^*(1)}^\dagger_{(f^e)^\iota}[p^m] \rightarrow \mu_{p^m}$. 

In the following, we assume:
\begin{itemize}
\item[(H.F)] $H^0(\QQ_p(\mu_{p^\infty}), T^\dagger)$ is finite
\item[(H.nT)] $H^0(\QQ_{p,\infty},T^\dagger) = 0$.
\end{itemize}

Let $A \in \{T^\dagger, T^\dagger_{f^e}, {T^*(1)}^\dagger, {T^*(1)}^\dagger_{(f^e)^\iota}\}$. {Since the absolute Galois group of $\mathbb{Q}_{p , \infty}$ acts trivially on $\Lambda^1/(f^e)$, (H.F) and (H.nT) are satisfied by $T^\dagger_{f^e}$, and (H.F) and (H.nT) also hold for ${T^*(1)}^\dagger$ by duality.
} 
The hypothesis (H.nT) together with the inflation-restriction sequence allow us to identify
$$
H^1(\QQ_{p,n},A) \simeq H^1(\QQ_{p,\infty},A)^{\Gamma^1_n}.
$$
Furthermore, the long exact sequence induced from $0 \rightarrow A[p^m] \rightarrow A  \xrightarrow{p^m} A \rightarrow 0$, together with hypothesis (H.nT), give
$$
H^1(\QQ_{p,n}, A[p^m]) \simeq H^1(\QQ_{p,n}, A)[p^m].
$$ 
For $I \subset \{1, \ldots , g\}$, we define $H^1_{I} (\QQ_{p,\infty}, T^\dagger)$ as the inverse image of $H^1_{I} (\QQ_p(\mu_{p^\infty}), T^\dagger)^\Delta$ by the restriction map 
$$
H^1(\QQ_{p,\infty}, T^\dagger) \rightarrow H^1(\QQ_p(\mu_{p^\infty}), T^\dagger)^\Delta,
$$
which is an isomorphism by hypothesis (H.F) and inflation-restriction sequence as $\Delta$ is of order $(p-1)$ and $H^0(\QQ_p(\mu_{p^\infty}), T^\dagger)$ is finite of order a power of $p$. We define also
$$
\begin{array}{rll}
H^1_{I} (\QQ_{p,\infty}, T^\dagger_{f^e}) &:= H^1_{I} (\QQ_{p,\infty}, T^\dagger) \otimes \Lambda^1/(f^e) & \subset H^1(\QQ_{p,\infty}, T^\dagger_{f^e}) \\
H^1_{I} (\QQ_{p,n}, T^\dagger_{f^e}) &:= H^1_{I} (\QQ_{p,\infty}, T^\dagger_{f^e})^{\Gal(\QQ_{p,\infty}/\QQ_{p,n})} &\subset  H^1 (\QQ_{p,n}, T^\dagger_{f^e}) \\
H^1_{I} (\QQ_{p,n}, T^\dagger_{f^e}[p^m]) &:= H^1_{I} (\QQ_{p,n}, T^\dagger_{f^e})[p^m] &\subset  H^1 (\QQ_{p,n}, T^\dagger_{f^e}[p^m]). 
\end{array}
$$
and similarly for the dual ${T^*(1)}^\dagger$.

For $n\ge0$, we write $F_n=\QQ_\infty^{\Gamma_n^1}$. Set 
$$
S_I( A/ F_n) := \ker\left(H^1(F_n, A) \rightarrow \prod\limits_{w\nmid p }\frac{ H^1(F_{n,w},A)}{H^1_f(F_{n,w},A)} \times  \frac{H^1(\QQ_{p,n},A)}{H^1_{I}(\QQ_{p,n},A)}\right),
$$
where $A \in \{T^\dagger, T^\dagger_{f^e}, T^\dagger[p^m], T^\dagger_{f^e}[p^m], {T^*(1)}^\dagger, {T^*(1)}^\dagger_{(f^e)^\iota}, {T^*(1)}^\dagger[p^m], {T^*(1)}^\dagger_{(f^e)^\iota}[p^m]\}$.
To simplify the notation, we shall write 
$$
\begin{array}{rl}
A & \text{for }   T^\dagger, T^\dagger_{f^e}, {T^*(1)}^\dagger \text{ or } {T^*(1)}^\dagger_{(f^e)^\iota} ,\\
A_m & \text{for }  A[p^m],\\
A^* & \text{for the dual of } A \, (\text{e.g. }{T^*(1)}^\dagger \text{ if } A={T}^\dagger ),\\
A^*_m & \text{for } A^*[p^m],
\end{array}
$$
and
$$
H^1_{\loc}(F_{n,v}, \bullet) = 
\begin{cases}
H^1_{f}(F_{n,v},\bullet) & \text{ if } v \nmid p,\\
H^1_{I}(F_{n,v},\bullet) & \text{otherwise. } 
\end{cases}  
$$
Here $H^1_f$ is the unramified subgroup for $v\nmid p$. 

\begin{lemma} \label{cardinaux}
Let $\chi_{\text{glob.},F_n}(A_m)$ be the global Euler characteristic of $A_m$ over $F_n$. Let $I \subset \{1, \ldots , g\}$ and write $I^c$ for its complement. Then,
$$
\#  S_I( A_m/ F_n) =\frac{ \#  S_{I^c}( A^*_m/ F_n)}{\chi_{\text{glob.},F_n}(A_m)\cdot[H^1(\QQ_{p,n}, A_m) : H^1_I (\QQ_{p,n}, A_m)]}.
$$
\end{lemma}
\begin{proof}
We only prove the lemma for $A = T^\dagger_{f^e}$, which would imply the other cases.
Let $\Sigma$ be a finite set of places of $\QQ$ containing all places where $T$ is ramified, all places over $p$ and the infinite place. We write $\Sigma(F_n)$ for the places of $F_n$ dividing those in $\Sigma$ and let $\QQ_\Sigma$ be the maximal extension of $\QQ$ unramified outside $\Sigma$. Then $S_I (A_m/F_n) \subset H^1(\QQ_\Sigma/F_n, A_m)$.

Set
$$
\begin{array}{rcl}
P^i_\Sigma := \prod_{v\in\Sigma(F_n)} H^i(F_{n,v},A_m) &\text{and}& P^{*,i}_\Sigma := \prod_{v\in\Sigma(F_n)} H^i(F_{n,v},A^*_m), \\
L_p :=  H^1_{I} (\QQ_{p,n}, A_m) &\text{and}& L_p^* := H^1_{I^c}(\QQ_{p,n}, A^*_m), \\
L_v :=  H^1_{f} (F_{n,v}, A_m) &\text{and}& L_v^* := H^1_{f}(F_{n,v}, A^*_m) \quad\text{for } v\nmid p,\\
L := \prod_{v\in\Sigma(F_n)} L_v &\text{and}& L^* := \prod_{v\in\Sigma(F_n)} L_v^* .
\end{array}
$$
Let
\[
\lambda^i : H^i(\QQ_\Sigma/F_n, A_m) \rightarrow P^i_\Sigma 
\]
be the restriction map and write
$$
\begin{array}{c}
G^i := \image \lambda^i , \quad K^i := \ker \lambda^i.
\end{array}.
$$
We have similarly $\lambda^{*,i}$, $G^{*,i}$, $K^{*,i}$ for $A^*$.
For all $v\in \Sigma(F_n)$ the orthogonal complement of $L_v$ under the local Tate pairing is $L^*_v$ (this is a classical result when $v \nmid p$ and Lemma~\ref{secOrth:lem2} when $v=p$).
Thus we have 
$$
\# S_I (A_m/F_n) = \# K^1 \cdot \# (G^1\cap L) = \# K^1 \cdot \# G^1 \cdot \# L \cdot \# (G^1\cdot L)^{-1}.
$$
Since $\# K^1 \cdot \# G^1 = \# H^1(\QQ_\Sigma/F_n, A_m)$, and by duality $\# (G^1\cdot L) = \# P^1_\Sigma / \# (G^{*,1}\cap L^*)$, we get 
$$
\# S_I (A_m/F_n) =\# H^1(\QQ_\Sigma/F_n, A_m) \cdot \# L \cdot  \# (G^{*,1}\cap L^*)/\# P^1_\Sigma.
$$
By (H.nT), we have 
$$
\begin{array}{rl}
\# H^1(\QQ_\Sigma/F_n, A_m) & = \chi_{\text{glob.},F_n}(A_m)^{-1} \cdot \# H^0(\QQ_\Sigma/F_n, A_m) \cdot \# H^2(\QQ_\Sigma/F_n, A_m) \\
& = \chi_{\text{glob.},F_n}(A_m)^{-1} \cdot\# H^2(\QQ_\Sigma/F_n, A_m).
\end{array}
$$
By global duality $\# K^{*,1} = \# K^2$, we obtain
$$
\# (G^{*,1}\cap L^*) = \#  S_{I^c}( A^*_m/ F_n) / \# K^{*,1}= \#  S_{I^c}( A^*_m/ F_n) / \# K^2.
$$
Thus, we have
\begin{equation} \label{equation1}
\# S_I (A_m/F_n) = \chi_{\text{glob.},F_n}(A_m)^{-1} \cdot\# H^2(\QQ_\Sigma/F_n, A_m)/\# K^2 \cdot \# L/\# P^1_\Sigma \cdot \#  S_{I^c}( A^*_m/ F_n).
\end{equation}
Now, the local Euler characteristic formula tells us that $\# H^1 (F_{n,v}, A_m)= \# H^0 (F_{n,v}, A_m) \cdot \#H^2 ( F_{n,v}, A_m)$. Also for $v$ not dividing $p$, $\# H^0 (F_{n,v}, A_m) = \# H^1_f (F_{n,v}, A_m)$, we deduce that
\begin{equation}  \label{equation2}
\begin{array}{rl}
 \# L/\# P^1_\Sigma & = \prod_{v \in \Sigma(F_n), v\nmid p} \# (H^1_{f} (F_{n,v}, A_m)/H^1 (F_{n,v}, A_m)) \times  \# (H^1_{I} (\QQ_{p,n}, A_m)/H^1 (\QQ_{p,n}, A_m)) \\
 & = \prod_{v \in \Sigma(F_n), v\nmid p} \# (H^0 (F_{n,v}, A_m)/H^1 (F_{n,v}, A_m)) \times  \# (H^1_{I} (\QQ_{p,n}, A_m)/H^1 (\QQ_{p,n}, A_m)) \\
 & = \prod_{v \in \Sigma(F_n), v\nmid p} \# (H^2 (F_{n,v}, A_m))^{-1} \times  \# (H^1_{I} (\QQ_{p,n}, A_m)/H^1 (\QQ_{p,n}, A_m)).
\end{array}
\end{equation}
On the other hand, global duality and (H.nT) tell us that $\# \coker \lambda^2 = \# H^0 (\QQ_\Sigma / F_n, A_m^*) = 1$, so
\begin{equation} \label{equation3}
\# H^2(\QQ_\Sigma/F_n, A_m)/\# K^2 = \# G^2 = \# P_\Sigma^2 / \coker \lambda^2 = \# P_\Sigma^2.
\end{equation} 
Furthermore, local Tate duality implies that $$\# H^2 ( F_{n,v}, A_m) =\# H^0 ( F_{n,v}, A_m^*)= 1.$$
Hence, the result follows on combining the equalities~\eqref{equation1}, \eqref{equation2} and \eqref{equation3} above.
\end{proof}

For the rest of this article, we assume in addition to (H.HT), (H.crys) and (H.Frob) that (H.P) holds. Under these hypotheses, we give a  generalization of \cite[Lemma~3.3]{kim08}.
\begin{lemma} \label{sec4:lemCtrl1}
Let $I \subset \{1, \ldots, g\}$ with $\#I = g_+=g_-$. Then $ \# S_I( A_m/ F_n) / \# S_{I^c}( A^*_m/ F_n)$ is bounded as $n$ and $m$ vary. 
\end{lemma}
\begin{proof}
We shall first of all compute the quantity $$\chi_{\text{glob.},F_n}(A_m)^{-1} \cdot [H^1(\QQ_{p,n}, A_m) : H^1_I (\QQ_{p,n}, A_m)]^{-1}$$ using Lemma~\ref{cardinaux}.

Since $F_n$ is totally real, one has $\chi_{\text{glob.},F_n}(A_m) = \# (A_m^-)^{-[F_n : \QQ]}$, where $A_m^-$ is the subgroup of $A_m$ on which complex conjugation acts by $-1$. As complex conjugation acts trivially on $\Lambda^1/(f^e)$, we get $\chi_{\text{glob.},F_n}(A_m) = p^{-m\cdot[F_n : \QQ]\cdot e \cdot \mathrm{deg}(f) \cdot g_-}$. Furthermore, we know that $\# H^1(\QQ_{p,n}, A_m) = p^{-m\cdot[\QQ_{p,n} : \QQ_p] \cdot e \cdot \mathrm{deg}(f) \cdot g}$.

It remains to compute $\# H^1_{I}(\QQ_{p,n}, A_m)$. By definition,
$$
H^1_{I}(\QQ_{p,n}, A_m) = \left(H^1_{I}(\QQ_{p,\infty}, T^\dagger )\otimes \Lambda^1 /(f^e)\right)^{\Gamma^1_n}[p^m].
$$
From hypothesis (H.F),
$$
\# H^1_{I}(\QQ_{p,\infty}, T^\dagger ) = \# H^1_{I}(\QQ_p(\mu_{p^\infty}), T^\dagger )^\Delta .
$$
If we write $(\cdot)^\vee$ for the Pontryagin dual, 
$$
H^1_{I}(\QQ_p(\mu_{p^\infty}), T^\dagger )^\Delta \simeq (\image \col_{T,I}^\Delta)^\vee .
$$
Proposition~\ref{secReview:prop1} gives an inclusion of $\image \col_{T,I}^\Delta$ into a free $\Lambda^1$-module  of rank $g_+$ (say $N_{I,\Delta}$) with finite index:
$$
0 \rightarrow \image \col_{T,I}^\Delta \rightarrow N_{I,\Delta} \rightarrow K_{I,\Delta} \rightarrow 0,
$$
which gives the long exact sequence
$$
\begin{array}{c}
0 \rightarrow (K_{I,\Delta}^\vee \otimes \Lambda^1/(f^e))^{\Gamma^1_n} \rightarrow (N_{I,\Delta}^\vee \otimes \Lambda^1/(f^e))^{\Gamma^1_n} \\ \rightarrow ((\image \col_{T,I}^\Delta)^\vee \otimes \Lambda^1/(f^e))^{\Gamma^1_n} 
\rightarrow H^1 (\Gamma^1_n,K_{I,\Delta}^\vee \otimes \Lambda^1/(f^e) ) .
\end{array}
$$
Since $H^1 (\QQ_{p,\infty}/\QQ_{p,n},K_{I,\Delta}^\vee \otimes \Lambda^1/(f^e) ) $ is finite it implies that
$$
\begin{array}{rl}
\#  ((\image \col_{T,I}^\Delta)^\vee [p^m] \otimes \Lambda^1/(f^e))^{\Gamma^1_n} & \leqslant C \cdot \# (N_{I,\Delta}^\vee \otimes \Lambda^1/(f^e))^{\Gamma^1_n} [p^m] \\
 & \leqslant C \cdot p^{m\cdot e \cdot \mathrm{deg}(f)\cdot g_+ \cdot p^n}
\end{array}
$$
where $C<\infty$ is independant of $n$ and $m$.
Hence the lemma follows.
\end{proof}

Note that our result is actually weaker than \cite[Lemma~3.3]{kim08}, which in fact gives an equality. However, our result is sufficient to prove the following generalization of Lemma~3.4 in \textit{op. cit.}

\begin{lemma} \label{sec4:lemCtrl2}
For all positive integers $n$ and $m$, the natural map
$$
S_I( A_m / F_n) \rightarrow S_I( A / F_n)[p^m]
$$
is injective and its cokernel is finite and bounded as $n$ and $m$ vary.
\end{lemma}
\begin{proof}
Consider the diagram 
\begin{center}
\begin{tikzcd}
0 \arrow{r} & S_I( A_m/ F_n) \arrow{d} \arrow{r} & H^1(\QQ_\Sigma/F_n, A_m) \arrow{d} \arrow{r} &  \prod\limits_{v \in \Sigma(F_n)} \frac{H^1(F_{n,v}, A_m)}{ H^1_{\loc}(F_{n,v}, A_m)} \arrow{d}{\prod f_v} \\
0 \arrow{r} & S_I( A/ F_n)[p^m] \arrow{r} & H^1(\QQ_\Sigma/F_n, A)[p^m] \arrow{r} & \prod\limits_{v \in \Sigma(F_n)} \frac{H^1(F_{n,v}, A)}{H^1_{\loc}(F_{n,v}, A)}.
\end{tikzcd}
\end{center}
We already know that the center vertical map is an isomorphism. Since $H^1_{I}(\QQ_{p,n}, A_m) = H^1_{I}(\QQ_{p,n}, A)[p^m]$, the map $f_p$ is injective.
The local condition is the unramified condition for $v \nmid p$, thus one has 
$$
\begin{array}{rl}
H^1(F_{n,v}, A)/H^1_{f}(F_{n,v}, A) & \subset H^1(F_{n,v}^{\text{ur}}, A) \\
H^1(F_{n,v}, A_m)/H^1_{f}(F_{n,v}, A_m) & \subset H^1(F_{n,v}^{\text{ur}}, A_m) 
\end{array}
$$
with $F_{n,v}^{\text{ur}}$ the maximal unramified extension of $F_{n,v}$.

The short exact sequence $0 \rightarrow A_m \rightarrow A \xrightarrow{p^m} A \rightarrow 0$ gives
$$
A^{I_v}/p^m A^{I_v} = \ker (H^1(F_{n,v}^{\text{ur}}, A_m) \rightarrow H^1(F_{n,v}^{\text{ur}}, A))
$$
and we have
$$
\# A^{I_v}/p^m A^{I_v}  \leqslant \# A^{I_v}/ (A^{I_v})_\text{div} < \infty.
$$
Since no prime splits completely in $\QQ_\infty/\QQ$, we conclude that $\ker \prod f_v$ is bounded as $n$ and $m$ vary. Applying the snake lemma, the lemma follows. 
\end{proof}

\begin{remark} \label{sec4:remarkIsotypic}
The results above are stated for $F=\QQ$ and the isotypic component of the trivial character of $\Delta$.
Let $\eta$ be a character on $\Delta$ and denote by $\overline{\eta}$ its complex conjugate and by $T^\dagger_\eta$ the $\Zp$-module $T^\dagger$ with action of the Galois group twisted by $\eta$, then 
$$
H^1(\QQ_p(\mu_{p^\infty}),T^\dagger)^{\overline{\eta}} = H^1(\QQ_p(\mu_{p^\infty}),T^\dagger_\eta)^\Delta = H^1(\QQ_{p,\infty},T^\dagger_\eta).
$$
The results above (for the isotypic component of the trivial character) in fact hold for every isotypic component because we may replace $T^\dagger$ and $T^{\dagger,*}$ by $T^\dagger_\eta$ and $T^{\dagger,*}_{\overline{\eta}}$ respectively in the proofs above.

Similarly, suppose that $F/\QQ$ is a general abelian extension in which $p$ is unramified with $[F:\QQ]$ coprime to $p$, we may prove the above results for an isotypic component corresponding to a character of $\Gal(F/\QQ)$.
\end{remark}

\section{Proof of functional equations}\label{S:proof}

We first review the following proposition of \cite{greenberg89}, which is a crucial ingredient of our proof for the functional equation.
\begin{proposition} \label{sec5:propGreenberg}
Let $X$ and $Y$ be two co-finitely generated $\Lambda^1$-modules such that
\begin{enumerate}
\item[i)] for all irreducible distinguished polynomial $f \in \Lambda^1$ and positive integer $e$,
$$
\operatorname{corank}_{\mathbb{Z}_p} (X \otimes \Lambda^1/(f^e))^{\Gamma^1} =\operatorname{corank}_{\mathbb{Z}_p} (Y \otimes \Lambda^1/(f^e))^{\Gamma^1} ;
$$
\item[ii)] for all positive integers $n$ and $m$, $\# X^{\Gamma^1_n}[p^m] / \# Y^{\Gamma^1_n}[p^m]$
is bounded as $n$ vary.
\end{enumerate}
Then, the Pontryagin duals of $X$ and $Y$ are pseudo-isomorphic.
\end{proposition}

We can now prove our main result.
\begin{theorem}\label{thm:main}
Let $I \subset \{1, \ldots , g\}$ of cardinality $\# I = g_+=g_-$ and $I^c$ its complement. Then $\Sel_I(T^\dagger/F(\mu_{p^\infty}))^\vee$ and $\Sel_{I^c}({T^*(1)}^\dagger/F(\mu_{p^\infty}))^{\vee,\iota}$ are pseudo-isomorphic $\Lambda$-modules.
\end{theorem}
\begin{proof}
It is enough to consider the isotypic component for the trivial character, as explained in Remark~\ref{sec4:remarkIsotypic}. 
Since $\Lambda^1/(f^e)$ is a free $\Zp$-modules on which $G_{\QQ_\infty}$ acts trivially, we have 
$$
\begin{array}{rl}
S_I(T^\dagger / \QQ_\infty) \otimes \Lambda^1/(f^e) & = \ker\left(H^1(\QQ_\Sigma/\QQ_\infty, T^\dagger)\otimes \Lambda^1/(f^e) \rightarrow \prod\limits_{w \in \Sigma} \frac{H^1(\QQ_{\infty,w},T^\dagger) \otimes \Lambda^1(f^e)}{H^1_{\loc}(\QQ_{\infty,w},T^\dagger) \otimes \Lambda^1/(f^e)}\right)\\
& = \ker\left(H^1(\QQ_\Sigma/\QQ_\infty, T^\dagger\otimes \Lambda^1/(f^e) ) \rightarrow \prod\limits_{w \in \Sigma} \frac{H^1(\QQ_{\infty,w},T^\dagger \otimes \Lambda^1(f^e))}{H^1_{\loc}(\QQ_{\infty,w},T^\dagger) \otimes \Lambda^1/(f^e)}\right).
\end{array}
$$

For the prime above $p$, we have by definition an injection
$$
 \frac{H^1(\Qp, T^\dagger_{f^e})}{H^1_I(\Qp, T^\dagger_{f^e})} \hookrightarrow \frac{H^1(\QQ_{p,\infty}, T^\dagger_{f^e})}{H^1_I(\QQ_{p,\infty}, T^\dagger) \otimes \Lambda^1/(f^e)} .
$$
For a prime $w$ of $\QQ_\infty$ above a prime $v\neq p$ in $\QQ$, 
since $\QQ_{\infty, w}/\QQ_v$ is an unramified $\Zp$-extension, we have the injection
$$
 \frac{H^1(\QQ_v, T^\dagger_{f^e})}{H^1_{f}(\QQ_v, T^\dagger_{f^e})} \hookrightarrow \frac{H^1(\QQ_{\infty,w}, T^\dagger_{f^e})}{H^1_{f}(\QQ_{\infty,w}, T^\dagger) \otimes \Lambda^1/(f^e)} .
$$
Since $H^1(\QQ, T^\dagger_{f^e}) \simeq H^1(\QQ_\infty, T^\dagger_{f^e})^{\Gamma^1}$, applying the snake lemma on the diagram
\begin{center}
\begin{tikzcd}
0 \arrow{r} & S_I( T^\dagger_{f^e}/ \QQ) \arrow{d} \arrow{r} & H^1(\QQ_\Sigma/\QQ, T^\dagger_{f^e}) \arrow{d} \arrow{r} &  \prod\limits_{v \in \Sigma} \frac{H^1(\QQ_{v}, T^\dagger_{f^e})}{ H^1_{\loc}(\QQ_{v}, T^\dagger_{f^e})} \arrow{d}{\prod f_v} \\
0 \arrow{r} & (S_I( T^\dagger/ \QQ_\infty)\otimes \Lambda^1/(f^e) \arrow{r})^{\Gamma^1} & H^1(\QQ_\Sigma/\QQ_\infty, T^\dagger_{f^e})^{\Gamma^1} \arrow{r} & \left(\prod\limits_{v \in \Sigma} \frac{H^1(\QQ_{\infty,v}, T^\dagger_{f^e})}{H^1_{\loc}(\QQ_{\infty,v}, T^\dagger) \otimes \Lambda^1/(f^e)}\right)^{\Gamma^1}
\end{tikzcd}
\end{center}
we get
$$
S_{I}(T^\dagger_{f^e}/\QQ) = (S_{I}(T^\dagger/\QQ_\infty) \otimes \Lambda^1/(f^e) )^{\Gamma^1}
$$
and similarly one has
$$
S_{I}(T^\dagger/F_n) = S_{I}(T^\dagger/\QQ_\infty)^{\Gamma^1_n}.
$$ 
Now by Lemmas~\ref{sec4:lemCtrl1} and \ref{sec4:lemCtrl2}, we see that $S_{I}(T^\dagger/\QQ_\infty)$ and $S_{I^c}({T^*(1)}^\dagger/\QQ_\infty)^\iota$ satisfy Proposition~\ref{sec5:propGreenberg} and the theorem follows.
\end{proof}

\section{Application to abelian varieties}\label{S:AV}
Let $A$ be an abelian variety of dimension $d$ defined over $F$ with good supersingular reduction at all primes above $p$. Denote by $T_p(A)$ its Tate module at $p$. As explained in the introduction,  $T_p(A)$ satisfies hypotheses (H.HT), (H.crys) and (H.Frob). One has $T_p(A)^\dagger \simeq A^\vee[p^\infty]$ where $A^\vee$ is the dual abelian variety. By the main result of \cite{imai75} hypothesis (H.F) is satisfied. Also, since $F$ is unramified over $\QQ$ and $p>2$,  of \cite[Lemma~5.11]{mazur72} implies that $A^\vee(F)$ has no $p$-torsion, and $F_\infty/F$ being a $\Zp$-extension, $A^\vee(F_\infty)$ has no $p$-torsion as well and (H.nT) is indeed satisfied. Finally, the Hodge structure of $A$ tells us that $g=[F:\QQ]\cdot 2d$ and $g_+ = [F :\QQ]\cdot d$. Thus Theorem~\ref{thm:main} applies in this setting. In particular, we have:

\begin{theorem}\label{thm:AV}
Let $I \subset \{1, \ldots , [F:\QQ]\cdot 2d\}$ of cardinality $\# I= [F:\QQ]\cdot d$ and $I^c$ its complement. Then $\Sel_I(A/F(\mu_{p^\infty}))^\vee$ and $\Sel_{I^c}(A^\vee/F(\mu_{p^\infty}))^{\vee,\iota}$ are pseudo-isomorphic $\Lambda$-modules.
\end{theorem}

Suppose now that $A$ is an abelian variety with complex multiplication over the CM field $F$ and write $\mathcal{O}_F$ for its ring of integers. We recall from \cite{buyukboduklei1} that, assuming the ``Perrin-Riou-Stark conjecture'' (Conjecture 4.18 of \textit{op. cit.}), the signed-Coleman maps allow us to define signed $p$-adic $L$-functions $\mathcal{L}^{I}(A) \in \mathcal{O}_F \lb \Gamma^1 \rb$. They  satisfy a suitable interpolation property (Proposition 7.15 of \textit{op. cit.}). One of the main results of \textit{op. cit.} (Theorem 7.16)  relates the characteristic ideal of the Pontrygain dual of
$$
\Sel_I(A/F_\infty) = \Sel_I(A/F(\mu_{p^\infty}))^\eta, \quad \text{where $\eta$ is the trivial character of $\Delta$,} 
$$
to the signed $p$-adic $L$-functions. That is 
\begin{equation}\label{eq:mainconjecture}
\operatorname{char}\left(\Sel_I(A/F_\infty)^\vee\right) = \frac{\mathcal{L}^{I}(A)}{(\gamma - 1)^{n(I)}} \cdot \mathcal{O}_F \lb \Gamma^1 \rb,
\end{equation}
where the integer $n(I)\geqslant 0$ depends on the image of the signed-Coleman map and is equal to $0$ when the basis $I$ is choosen to be strongly admissible in the sense of \cite[Definition 3.2]{buyukboduklei0}.

Theorem~\ref{thm:AV} above tells us that
$$
\operatorname{char}\left(\Sel_I(A/F_\infty)^\vee\right) = \operatorname{char}\left(\Sel_{I^c}(A^\vee/F_\infty)^{\iota,\vee}\right).
$$
Thus, we deduce a functional equation for the (conjectural) signed $p$-adic $L$-functions 
$$
\frac{\mathcal{L}^{I}(A)}{(\gamma - 1)^{n(I)}} \sim \iota\left( \frac{\mathcal{L}^{I^c}(A^\vee)}{(\gamma - 1)^{n(I^c)}} \right).
$$
Here, $\alpha\sim\beta $ for $\alpha,\beta\in \mathcal{O}_F \lb \Gamma^1 \rb$ means that $\alpha$ and $\beta$ differ by a unit in $ \mathcal{O}_F \lb \Gamma^1 \rb$.

In \cite{sprung16}, Sprung proved a similar functional equation for supersingular elliptic curves without the CM condition nor  the main conjecture (that is, the equation \eqref{eq:mainconjecture}).  It would be interesting to study whether the same can be done for general supersingular abelian varieties.

\bibliography{references}{}
\bibliographystyle{amsalpha}
\end{document}